\documentclass[11pt]{amsart}
\usepackage{geometry}
\usepackage{natbib}
\usepackage{ulem}

\usepackage{graphicx}
\usepackage{amssymb}
\usepackage{amsmath}
\usepackage{amssymb}
\usepackage{epstopdf}
\usepackage{color}
 \usepackage[colorlinks=true]{hyperref}
\hypersetup{urlcolor=blue, citecolor=blue}

\usepackage{setspace}
\newtheorem{theorem}{Theorem}[section]
\newtheorem{lemma}[theorem]{Lemma}
\newtheorem{remark}[theorem]{Remark}

\newtheorem{assumption}[theorem]{Assumption}

\title{Posterior contraction rates for non-parametric state and drift estimation}
\author{Sebastian Reich and Paul Rozdeba}
\thanks{Institute of Mathematics, University of Potsdam, Karl-Liebknecht-Str.~24/25, D-14476 Potsdam, Germany, e-mail: \tt{sebastian.reich@uni-potsdam.de}}
\date{\today}                                           

\begin{document}
\maketitle

%
%
%
%
%


%
%
%

\begin{abstract} We consider a combined state and drift estimation problem for the linear stochastic heat equation.
The infinite-dimensional Bayesian inference problem is formulated in terms of the Kalman--Bucy filter over an extended state space, and its long-time asymptotic properties are studied. Asymptotic posterior contraction rates in the unknown drift function are the main contribution of this paper. Such rates have been studied before for stationary non-parametric Bayesian inverse problems, and here we demonstrate the consistency of our time-dependent formulation with these previous results building upon scale separation and a slow manifold approximation.
\end{abstract}


%
%
\section{Introduction}
%
%

In this paper, we consider the combined state and drift estimation problem for the stochastic heat equation
\begin{equation}\label{eq:SEE}
{\rm d} u_t = \partial_x^2 u_t{\rm d}t + f^\ast{\rm d}t + \gamma^{1/2}{\rm d}w_t
\end{equation}
over the domain $\Omega = (0,\pi)$ and for $t\ge 0$, with zero Dirichlet boundary conditions $u_t(0) = u_t(\pi) = 0$ and zero initial condition $u_0(x) = 0$ for all $x \in \Omega$, without loss of generality. We denote by $w_t$ a cylindrical Wiener process where $\gamma >0$ characterises the strength of the model error. 
The unknown drift function $f^\ast$ is assumed to belong to a  Sobolev space of appropriate regularity. We specify the properties of $w_t$ and $f^*$ more precisely in Section \ref{sec:math} below.

Linear stochastic partial differential equations (SPDEs) of the form (\ref{eq:SEE}) are well understood from analytical and numerical perspectives. See, for example, \cite{DaPrato_SPDE_book,Hairer2009_notes,LordSPDEbook}. In this paper, we  specifically focus on the estimation of the states $u_t$, $t>0$, as well as the unknown drift function $f^\ast$ from noisy state measurements given by
\begin{equation} \label{eq:measurement}
{\rm d}y_t = u_t{\rm d}t + \rho^{1/2} {\rm d}v_t,
\end{equation} 
where $y_0(x) = 0$ for all $x\in \Omega$, $\rho>0$ denotes the strength of the measurement noise, and $v_t$ is another cylindrical Wiener process independent of the model error $w_t$.

Our focus is on asymptotic posterior contraction rates, which have been widely studied in the context of non-parametric
stationary inverse problems. See, for example, \cite{knapik2011,GineNickl2016book,Ghosal2017book}. We note that pure parameter estimation problems for SPDEs given exact observations of the states are also well-studied. See, for example,  \cite{Cialenco2018_review} for a recent survey. First steps towards non-parametric time-dependent inverse problems have been taken in \cite{YanThesis2019}. However, state estimation for a given drift function $f^\ast$ is framed as an inference problem over a fixed time interval resulting in a smoothing problem, whereas filtering problems are the focus of this paper. Conversely, estimation of a drift function $f^\ast$ is analyzed under the assumption of perfectly observed model states. Similarly, parameter estimation for a class of linear SPDEs from noiseless local state measurements has been considered in \cite{Altmeyer2019}, whereas we assume noisy observations of the full states throughout this paper.

In this paper, the infinite-dimensional, time-dependent combined state and drift estimation problem is formulated and analysed in terms of the well-known Kalman--Bucy filter equations for the posterior mean and covariance operator in an augmented state space. See, for example, \cite{sr:jazwinski,sr:simon,curtain1975} for an introduction to the Kalman--Bucy filter in finite and infinite dimensions. More recently, the Kalman--Bucy filter has been reformulated as a set of mean-field equations termed the ensemble Kalman--Bucy filter. See, for example \cite{sr:br11,sr:dWRS18,nrr2019}. These mean-field equations allow for a concise formulation of our time-dependent estimation problem. Furthermore, although not considered in this paper, the Kalman--Bucy mean-field equations can be generalised to non-Gaussian estimation problems \citep{sr:meyn13,sr:TdWMR17}. A key observation arising from the analysis of the Kalman--Bucy mean-field equations is the time-scale separation in the dynamics of the state and parameter variables, allowing us to apply the concept of slow manifolds \citep{verhulst07}.

The remainder of this paper is structured as follows. A mathematical formulation of the combined state and drift estimation problem in terms of Fourier modes is provided in Section \ref{sec:math}. Furthermore, the time-dependent estimation problem is formulated in terms of Kalman--Bucy mean-field equations. This formulation is applied to a stationary drift estimation problem in Section \ref{sec:drift}. Asymptotic posterior contraction rates are then derived and compared to existing results for a closely related non-parametric Bayesian inference problem. These results are extended to the combined state and drift estimation problem in Section \ref{sec:state_drift}. The essential analysis of the single-mode Kalman--Bucy filter equations is carried out in Section \ref{sec:single_mode}. A numerical exploration of the combined state and parameter estimation problem is carried out in Section \ref{sec:numerics}. The paper concludes with Section \ref{sec:conclusions}.

%
%
\section{Mathematical problem formulation} \label{sec:math}
%
%

In this paper, we analyse the state and drift estimation problem using a spectral representation in terms of Fourier modes. We provide the essential background in this section. It is well known that solutions of (\ref{eq:SEE}) can be expanded in Fourier modes $\{\mathcal{U}_t(k)\}_{k\ge 1}$, that is,
\begin{equation}
u_t(x) = \sum_{k\ge 1} \mathcal{U}_t(k) \sin (kx).
\end{equation} 
Because of (\ref{eq:SEE}), the Fourier modes obey the stochastic differential equations (SDEs)
\begin{equation} \label{eq:model1}
{\rm d}\mathcal{U}_t(k) = -k^2 \mathcal{U}_t(k) {\rm d}t + \mathcal{F}^\ast (k){\rm d}t + \gamma^{1/2} {\rm d}\mathcal{W}_t(k), \qquad k\ge 1,
\end{equation}
where 
\begin{equation}
w_t(x) = \sum_{k\ge 1} \mathcal{W}_t(k) \sin (kx)
\end{equation}
denotes the Fourier representation of the cylindrical Wiener process $w_t$.
Here $\mathcal{F}^\ast(k)$, $k\ge 1$, denote the Fourier coefficients of $f^\ast$, and $\mathcal{W}_t(k)$ independent standard Brownian motions. The initial conditions are $\mathcal{U}_0(k) =0$ for all $k\ge 1$. Therefore, (\ref{eq:SEE}) can be viewed as a stochastic evolution equation on the separable Hilbert space of square integrable functions $u:\Omega \to \mathbb{R}$  of sufficient spatial regularity with zero boundary conditions. 

The measurement model (\ref{eq:measurement}) can also be transformed into Fourier space, yielding
\begin{equation}\label{eq:model2}
{\rm d}\mathcal{Y}_t(k) = \mathcal{U}_t(k){\rm d}t + \rho^{1/2} {\rm d}\mathcal{V}_t(k),\qquad k\ge 1,
\end{equation}
with $\mathcal{V}_t(k)$ representing independent standard Brownian motions, independent of the model errors $\mathcal{W}_t(k')$ for all $k'$, that is,
\begin{equation} \label{eq:obs_value_t}
v_t(x) = \sum_{k\ge 1} \mathcal{V}_t(k) \sin (kx).
\end{equation}
Please also note that measurement processes defined by (\ref{eq:measurement}) and (\ref{eq:model2}), respectively, are related by
\begin{equation}
y_t(x) = \sum_{k\ge 1} \mathcal{Y}_t(k) \sin (kx).
\end{equation}
Throughout this paper, we will exclusively work with the formulation (\ref{eq:model1}) and (\ref{eq:model2}) of our state and drift estimation problem in Fourier space. We now proceed to formulate a continuous-time Bayesian inference framework for this problem, which is used to derive asymptotic posterior contraction rates in Section \ref{sec:state_drift}.

Let us denote the joint posterior distribution in the state $\mathcal{U}_t(k)$ and the drift $\mathcal{F}^\ast$, 
conditioned on the observed data $\mathcal{Y}_{(0,t]}(k)$ up to time $t$, $k\ge 1$, by $\pi_t^{(k)}$. 
Here $\mathcal{Y}_{(0,t]}(k)$ represents a path of the observation process up to time $t>0$, whereas $\mathcal{Y}_{t}(k)$ in (\ref{eq:obs_value_t}) denotes its instantaneous value at time $t$. By construction, the equivalent process in function space is given by
\begin{align}
  y_{(0,t]}(x) = \sum_{k \geq 1} \mathcal{Y}_{(0,t]}(k) \sin(kx).
\end{align}
We emphasize that this joint posterior distribution is Gaussian due to the linearity of (\ref{eq:model1}) and (\ref{eq:model2}), respectively.

We next reformulate this Bayesian inference problem in terms of pairs of new random variables $(U_t(k),F_t(k))$, $k\ge 1$,
which satisfy the mean-field Kalman--Bucy equations \citep{sr:dWRS18,nrr2019}
\begin{subequations} \label{eq:MFKB}
\begin{align}
{\rm d}U_t(k) &= - k^{2}U_t(k){\rm d}t + F_t(k){\rm d}t + \gamma^{1/2} {\rm d}W_t(k) - K_t^u(k) {\rm d} I_t(k), \\
{\rm d} F_t(k) &= - K_t^f(k) {\rm d}I_t(k),
\end{align}
\end{subequations}
for a given observation path $\mathcal{Y}_{(0,t]}(k)$ with the innovation $I_t(k)$ that solves
\begin{equation}
{\rm d}I_t(k) := \frac{1}{2} \left(U_t(k) + \mathbb{E}[U_t(k)|\mathcal{Y}_{(0,t]}(k)]\right){\rm d}t - {\rm d}\mathcal{Y}_t(k),
\end{equation}
and with Kalman gains
\begin{equation}
K_t^u(k) := \rho^{-1} \mathbb{E}[\Delta U_t(k)^2|\mathcal{Y}_{(0,t]}(k)], \quad K_t^f(k) := \rho^{-1} \mathbb{E}[\Delta U_t(k) \Delta F_t(k)|\mathcal{Y}_{(0,t]}(k)],
\end{equation}
where 
\begin{equation}
\Delta U_t(k) := U_t(k) - \mathbb{E}[U_t(k)|\mathcal{Y}_{(0,t]}(k)]
\end{equation} 
and 
\begin{equation}\Delta F_t(k) := F_t(k) - \mathbb{E}[F_t(k)|\mathcal{Y}_{(0,t]}(k)].
\end{equation}

The stochastic process defined by the  Kalman--Bucy equations (\ref{eq:MFKB}) satisfies
\begin{equation}
(U_t(k),F_t(k)) \sim \pi_t^{(k)}.
\end{equation}
In other words, upon defining the conditional posterior expectation value
\begin{equation}
\pi^{(k)}_t[g] = \int_{\mathbb{R}^2} g(u,f)\,\pi^{(k)}_t(u,f)\,{\rm d}u\,{\rm d}f
\end{equation}
for any measurable function $g(u,f)$, it holds, for example, that the mean 
of $U_t(k)$, denoted by $\mathbb{E}[U_t(k)|\mathcal{Y}_{(0,t]}(k)]$, is equal to $\pi_t^{(k)}[u]$, and 
\begin{equation}
K_t^u(k) = \rho^{-1} \pi_t^{(k)}\left[ \left(u-\pi_t^{(k)}[u]\right)^2\right] = \frac{\sigma_t^u(k)}{\rho}
\end{equation}
where $\sigma_t^u(k)$ is the variance of $U_t(k)$. We will use the shorthand $\mathbb{E}[U_t(k)]$ instead of
$\mathbb{E}[U_t(k)|\mathcal{Y}_{(0,t]}(k)]$ from now on in order to simplify notations.

The prior assumptions are that $U_0(k) = 0$ almost surely, and that the 
$F_0(k)$, $k\ge 1$, are independent Gaussian random variables with mean zero and variance 
\begin{equation} \label{eq:prior_f}
\sigma_0^f(k) = k^{-2\alpha - 1}
\end{equation} 
for a suitable $\alpha > 0$. Since (\ref{eq:model1}) and (\ref{eq:model2}) are linear in the unknowns, $U_t(k)$ and $F_t(k)$ remain Gaussian under the evolution equations (\ref{eq:MFKB}), and we investigate their behaviour in terms of their mean and covariance for $t \gg 1$. We further assume that the true drift function $f^\ast$ has Sobolev regularity $\beta > 0$, that is
\begin{equation} \label{eq:drift_sobolev}
\sum_{k\ge 1} \mathcal{F}^\ast(k)^2 k^{-2\beta} < \infty.
\end{equation}
These choices correspond to a standard setting for Bayesian non-parametric inference \citep{knapik2011,GineNickl2016book,Ghosal2017book}.
For example, (\ref{eq:drift_sobolev}) is satisfied if
\begin{equation} \label{eq:drift_regularity}
\mathcal{F}^\ast(k) = c_k k^{-\beta - 1/2 - \delta}
\end{equation}
for any $\delta >0$ and any sequence of coefficients $c_k$ with bounded $|c_k|$ as $k\to \infty$. These coefficients can, for example, be realisations of i.i.d.\ uniform random variables from the interval $[-1,1]$.

In addition to an analysis of the continuous-time Bayesian inference problem (\ref{eq:MFKB}), we also study the (frequentist) 
dependence of the estimators (Bayesian posterior means)
\begin{equation} \label{eq:estimators}
\overline{U}_t(k) := \mathbb{E}[U_t(k)] = \pi_t^{(k)}[u], \qquad
\overline{F}_t(k) := \mathbb{E}[F_t(k)] = \pi_t^{(k)}[f],
\end{equation}
on the random observation process $\mathcal{Y}_{(0,t]}(k)$ in Section \ref{sec:state_drift}. We denote the expectation values of (\ref{eq:estimators}) with respect to $\mathcal{Y}_{(0,t]}(k)$ by
\begin{equation}
m_t^u(k):=\mathbb{E}^\ast [\overline{U}_t(k)] \quad  \mbox{and} \quad m_t^f(k):=\mathbb{E}^\ast [\overline{F}_t(k)]. 
\end{equation}
These two quantities characterise the systematic bias in the estimators (\ref{eq:estimators}), while the implied variances 
\begin{equation}
p_t^u (k) := \mathbb{E}^\ast \left[ \left(\overline{U}_t(k)-m_t^u(k)\right)^2\right], \qquad 
p_t^f(k) := \mathbb{E}^\ast \left[ \left(\overline{F}_t(k)-m_t^f(k)\right)^2\right]
\end{equation}
are a measure of the \lq frequentist\rq{} uncertainty of the estimators (\ref{eq:estimators}).

According to Lemma 8.2 from \cite{Ghosal2017book}, the posterior contraction rate $\phi_t(k)$ in
the $k$th Fourier mode of the drift estimator $F_t(k)$, that is, $\phi_t(k)$ in
\begin{equation} \label{eq:posterior_contraction}
\mathbb{E}^\ast \left[ \mathbb{P}\left\{F_t(k):|F_t(k)-\mathcal{F}^\ast(k)| \ge M_t \phi_t(k)\,|\,\mathcal{Y}_{(0,t]}(k) \right\}\right] \xrightarrow{t \to \infty} 0,
\end{equation}
is provided by
\begin{equation}
\phi_t(k)^2 = \sigma_t^f(k) + p_t^f(k) + (m_t^f(k)-\mathcal{F}^\ast(k))^2.
\end{equation}
Equation (\ref{eq:posterior_contraction}) holds for any nonnegative function $M_t$ with $M_t\to \infty$. Here, $\sigma_t^f(k)$ denotes the posterior variance of $F_t(k)$.

The main contributions from our subsequent analysis consist of providing bounds for the three quantities $\sigma_t^f(k)$, 
$p_t^f(k)$, and $m_t^f(k)-\mathcal{F}^\ast$ in terms of the statistical models (\ref{eq:model1}) and (\ref{eq:model2}), and additionally studying the asymptotic behaviour of the $l_2$ asymptotic contraction rate $\phi_t$, defined by
\begin{equation} \label{eq:l2_posterior_contraction_rate}
\phi_t^2 := \sum_{k\ge 1} \phi_t(k)^2,
\end{equation}
in terms of the Sobolev regularity $\beta$ of the true drift function $f^\ast$ and the variance of the prior characterised by $\alpha>0$ in (\ref{eq:prior_f}). Here, the $l_2$ asymptotic contraction rate $\phi_t$ is to be understood
in the sense that
\begin{equation} 
\mathbb{E}^\ast \left[ \mathbb{P}\left\{ F_t: \sum_{k\ge 1}  \left( F_t(k)-\mathcal{F}^\ast(k)\right)^2 \ge M_t^2 \phi_t^2\,|\, \{\mathcal{Y}_{(0,t]}(k)\}_{k \geq 1} \right\}\right] \xrightarrow{t \to \infty} 0
\end{equation}
for any nonnegative function $M_t$ with $M_t \to \infty$. In physical space,
\begin{equation}
\mathbb{E}^\ast \left[ \mathbb{P}\left\{ f_t: \frac{1}{\pi}
\int_\Omega \left( f_t(x)-f^\ast(x)\right)^2 {\rm d} x \ge M_t^2 \phi_t^2\,|\,y_{(0,t]} \right\}\right] \xrightarrow{t \to \infty} 0,
\end{equation}
where $f_t$ denotes the inverse Fourier transform of $F_t$, that is,
\begin{equation}
f_t(x) = \sum_{k\ge 1} F_t(k) \sin (kx).
\end{equation}

We first discuss a stationary formulation of the drift estimation problem in Section \ref{sec:drift}, which is closely linked to available results for non-parametric Bayesian inverse problems \citep{knapik2011,GineNickl2016book,Ghosal2017book}. The combined state and drift estimation problem is analysed in Section \ref{sec:state_drift}. 

\begin{remark}
A more general class of time-scaled prior variances over the unknown set of static parameters $\mathcal{F}^*(k)$ has been considered in
\cite{knapik2011}. This more general class of prior distributions could be incorporated into our analysis as well. However, we restrict the subsequent analysis to the simpler case (\ref{eq:prior_f}). 
\end{remark}

%
%
\section{Stationary drift estimation problem} \label{sec:drift}
%
%

Before addressing the full time-dependent problem, we first consider the stationary drift estimation problem in which $\gamma =0$ and ${\rm d}\mathcal{U}_t(k) = 0$ in (\ref{eq:model1}), giving rise to the forward model
\begin{equation}
\mathcal{U}^*(k) = k^{-2} \mathcal{F}^*(k)
\end{equation}
and observations
\begin{subequations}\label{eq:forward_model}
\begin{align}
{\rm d}\mathcal{Y}_t(k) &= \mathcal{U}^\ast(k){\rm d}t + \rho^{1/2}{\rm d}\mathcal{V}_t(k)\\
&= k^{-2} \mathcal{F}^\ast(k) {\rm d}t + \rho^{1/2} {\rm d}\mathcal{V}_t(k)
\end{align}
\end{subequations}
for $t \geq 0$. The forward model (\ref{eq:forward_model}) leads to a mildly ill-posed inverse problem. We follow a Bayesian approach by placing a Gaussian mean-zero prior with variance given by (\ref{eq:prior_f}) on the Fourier coefficients of the unknown drift function. Let us summarise our assumptions for ease of reference in the following:

\begin{assumption} \label{assumption1}
We assume that the Fourier coefficients of the true drift function satisfy (\ref{eq:drift_sobolev}).
The prior coefficients, denoted by $F_0(k)$, are independent Gaussian random variables with mean zero, that is, $\overline{F}_0(k) =0$, and with variance $\sigma_0^f(k)$ satisfying (\ref{eq:prior_f}).
\end{assumption}

Let us denote the Bayesian posterior estimate of the drift at time $t>0$ by $F_t(k)$, $k\ge 1$. We recall that $F_t(k)$ is Gaussian for all $t> 0$ provided $F_0(k)$ is Gaussian. Additionally, recall that we denote the mean of $F_t(k)$ by
$\overline{F}_t(k)$ and the variance by $\sigma_t^f(k)$. The Kalman--Bucy evolution equations (\ref{eq:MFKB}) reduce to the following equations for the mean and the variance of the $k$th Fourier mode:
\begin{subequations} \label{eq:KB_f}
\begin{align} \label{eq:KB_f_mean}
{\rm d} \overline{F}_t(k) &= - K_t^f(k) \left( k^{-2} \overline{F}_t(k){\rm d}t - {\rm d}\mathcal{Y}_t(k)\right),\\
\frac{\rm d}{{\rm d}t} \sigma_t^f(k) &= -K_t^f(k) \frac{\sigma_t^f(k)}{k^2}, \label{eq:KB_f_variance}
\end{align}
\end{subequations}
where $K_t^f(k)$ denotes the Kalman gain factor
\begin{equation}
K_t^f(k) := \frac{\sigma_t^f(k)}{\rho k^2}.
\end{equation}
The initial conditions are provided by Assumption \ref{assumption1}.

\begin{remark}
Alternatively, the Kalman--Bucy equations (\ref{eq:KB_f}) for the mean and the variance can be reformulated as mean-field equations in $F_t(k)$ directly, that is,
\begin{equation}
{\rm d}F_t(k) = - K_t^f(k) \left( \frac{1}{2k^2}\left( F_t(k) + \overline{F}_t(k)\right){\rm d}t - {\rm d}\mathcal{Y}_t(k) \right).
\end{equation}
These equations are a special case of (\ref{eq:MFKB}), and provide the starting point for numerical implementation in the form of the ensemble Kalman--Bucy filter \citep{sr:br11}, as well as extensions to nonlinear and non-Gaussian Bayesian inference problems \citep{sr:meyn13,sr:TdWMR17}.
\end{remark}

The evolution equation (\ref{eq:KB_f_variance}) for the posterior variance has the closed form solution
\begin{equation}\label{eq:sigma_f_drift}
\sigma_t^f(k) =
\frac{\sigma_0^f(k)}{\rho^{-1} k^{-4} \sigma_0^f(k) t +1}.
\end{equation}
The trace of the covariance of the full joint posterior process thus satisfies the asymptotic estimate
\begin{equation} \label{eq:lemma1_1}
\sum_{k\ge 1} \sigma_t^f(k) \asymp t^{-2\alpha/(2\alpha+5)}
\end{equation}
with the initial variances satisfying (\ref{eq:prior_f}). This result follows from asymptotic bounds for infinite sums. See
Lemma K.7 in \cite{Ghosal2017book} in particular.

Next, we carry out a \lq frequentist\rq{} analysis of the mean $\overline{F}_t(k)$ in terms of its bias with respect to the true $\mathcal{F}^\ast(k)$, and its variance $p_t^f(k)$ with respect to the measurement noise $\mathcal{V}_t(k)$. We rewrite (\ref{eq:KB_f_mean}) as
\begin{equation}
{\rm d} \overline{F}_t(k) = - K_t^f(k) \left( k^{-2}(\overline{F}_t(k)- \mathcal{F}^\ast(k)){\rm d}t
- \rho^{1/2} \mathrm{d} \mathcal{V}_t(k) \right).
\end{equation}
Denoting the expectation value of $\overline{F}_t(k)$ with respect to the measurement noise $\mathcal{V}_t(k)$ by $m_t^f(k)$, we obtain
\begin{equation}
\frac{\rm d}{{\rm d}t} m_t^f(k) = -k^{-2} K_t^f (k)\left( m_t^f(k) - \mathcal{F}^\ast (k)\right)
\label{eq:freqmean_ode}
\end{equation}
with initial condition $m_0^f(k) = 0$. Equation \eqref{eq:freqmean_ode} has the closed-form solution 
\begin{equation} \label{eq:mean_f_drift}
m_t^f(k) = \mathcal{F}^\ast(k) - \frac{\sigma_t^f(k)}{\sigma_0^f(k)}\mathcal{F}^\ast(k).
\end{equation}
Hence, based on (\ref{eq:drift_regularity}), the $l_2$-norm of the frequentist bias satisfies the asymptotic estimate
\begin{equation} \label{eq:lemma1_2}
\sum_{k\ge 1} \left(m_t^f(k)-\mathcal{F}^\ast(k)\right)^2 \asymp t^{-2\beta /(2\alpha +5)}.
\end{equation}
This result follows again from asymptotic bounds for infinite sums. 

We finally investigate the time evolution of the frequentist variance $p_t^f(k)$.  Starting from the stochastic differential equation
\begin{equation}
{\rm d} \left( \overline{F}_t(k)-m_t^f(k)\right) =
- K_t^f(k) \left( k^{-2} \left( \overline{F}_t(k)-m_t^f(k)\right){\rm d}t - \rho^{1/2} {\rm d}
\mathcal{V}_t(k) \right),
\end{equation}
using It\^{o}'s lemma yields
\begin{subequations} \label{eq:frequentist_variance}
\begin{align}
\frac{\rm d}{{\rm d}t} p_t^f(k) &= -\frac{2}{k^{2}} K_t^f(k) p_t^f(k) + \rho \left(K_t^f(k)\right)^2\\
&= -2\frac{\sigma_t^f(k)}{\rho k^4} p^f_t(k) + \frac{\left( \sigma_t^f(k)\right)^2}{\rho k^4}.
\end{align}
\end{subequations}
The initial variance is $p_0^f(k) = 0$ for all $k\ge 1$. In order to analyse the solution behaviour of (\ref{eq:frequentist_variance}), we introduce $\Delta p_t^f(k) := \sigma_t^f(k)-p_t^f(k)$ and find the associated evolution equation
\begin{equation}\label{eq:frequentist_variance_2}
\frac{\rm d}{{\rm d}t} \Delta p_t^f(k) = -2 \frac{\sigma_t^f(k)}{\rho k^4} \Delta p_t^t(k)
\end{equation} 
from which we conclude that $p_t^f(k) < \sigma_t^f(k)$, and
\begin{equation}
\lim_{t\to \infty} p_t^f(k) = \sigma_t^f(k).
\end{equation}
In fact, the explicit solution of (\ref{eq:frequentist_variance}) is 
\begin{equation}
p_t^f(k) =  \frac{\rho k^4 \sigma_0^f(k)^2 t}{\left( \rho k^4 +\sigma_0^f(k) t\right)^2}.
\end{equation}

The \lq frequentist\rq{} uncertainty is characterised by
\begin{equation} 
\overline{F}_t(k)-\mathcal{F}^\ast (k) = 
\mathcal{O}_{\rm P}(\Delta_t(k)), \qquad \Delta_t(k)\sim {\rm N}(m_t^f(k)-\mathcal{F}^\ast (k),
p_t^f(k)),
\end{equation}
so the probability of the event
\begin{equation}
|\overline{F}_t(k)-\mathcal{F}^\ast(k)| > M_t|\Delta_t(k)|
\end{equation}
has probability tending to zero for any nonnegative function $M_t$ with $M_t \to\infty$.  Here ${\rm N}(m,\sigma)$ denotes the Gaussian distribution with mean $m$ and variance $\sigma$.

\begin{theorem} \label{theorem1}
Under Assumption \ref{assumption1}, the posterior contraction rate (\ref{eq:l2_posterior_contraction_rate}) is given by
\begin{equation} \label{eq:contraction_rate_pure_drift}
\phi_t^2 = \sum_{k\ge 1} \left(
\sigma_t^f(k) + p_t^f(k) + (m_t^f(k)-\mathcal{F}^\ast (k))^2 \right)\asymp t^{-2\min\{\alpha,\beta\}/(2\alpha + 5)}
\end{equation}
for $t\to \infty$.
\end{theorem}

\begin{proof} According to Lemma 8.2 from \citep{Ghosal2017book}, the posterior contraction rate $\phi_t(k)$ in each Fourier mode is provided by 
\begin{equation}
\phi_t(k)^2 = \sigma_t^f(k) + p_t^f(k) + (m_t^f(k)-\mathcal{F}^\ast (k))^2.
\end{equation}
Because of (\ref{eq:lemma1_1}) and (\ref{eq:lemma1_2}) together with $p_t^f(k) < \sigma_t^f(k)$, the result (\ref{eq:contraction_rate_pure_drift}) follows. 
\end{proof}

\begin{remark} The following white noise forward model has been investigated in
\citep{knapik2011}:
\begin{equation}
\mathcal{Y}_n(k) = k^{-p} F^\ast (k) + \frac{1}{n^{1/2}} \Xi_n(k), \quad n \in \mathbb{Z}^+,
\end{equation}
where $\Xi_n(k)$ are i.i.d.~Gaussian random variables with mean zero and variance one. 
The associated inference problem corresponds to a sequence of measurements with measurement error decreasing as $1/n$.  In our continuous-time problem, we obtain the same asymptotic rates as in the case considered above with $p=2$ under the formal equivalence $t = n \to \infty$.
\end{remark}

%
%
\section{Time-dependent state and drift estimation} \label{sec:state_drift}
%
%

We now return to the full dynamic model (\ref{eq:model1}) subject to observations (\ref{eq:model2}). We primarily wish to estimate the drift function $\mathcal{F}^\ast$. However, because of the stochastic model errors in (\ref{eq:model1}), we also need to estimate the states $\mathcal{U}_t$. We start with a careful analysis of the single mode system for both small- and large-$k$ Fourier modes.

%
%
\subsection{Analysis of the single-mode filtering problem} \label{sec:single_mode}
%
%

In this section, we conduct a careful analysis of the single-mode filtering and parameter estimation problem.
We suppress the dependence on the mode number $k$, and introduce the parameter $\epsilon = k^{-2}$. 
The signal process is therefore given by
\begin{equation} \label{eq:signal_process}
{\rm d}\mathcal{U}_t = - \epsilon^{-1} \mathcal{U}_t {\rm d}t + \mathcal{F}^\ast{\rm d}t + \gamma^{1/2} {\rm d} \mathcal{W}_t
\end{equation}
with given initial condition $\mathcal{U}_0 = 0$ almost surely.  Observations of the process are given by
\begin{equation} \label{eq:obs_process_2}
{\rm d}\mathcal{Y}_t = \mathcal{U}_t {\rm d}t + \rho^{1/2} {\rm d}\mathcal{V}_t 
\end{equation}
with $\mathcal{Y}_0 = 0$ almost surely. The complete observation path up to time $t$ is denoted by $\mathcal{Y}_{(0,t]}$.

Recall that the mean-field Kalman--Bucy equations are given by (\ref{eq:MFKB}) and that their solutions $(U_t,F_t)$ are Gaussian 
distributed. We again drop the dependence on the Fourier mode number $k$ in the subsequent analysis. Taking conditional
expectations in the mean-field equations (\ref{eq:MFKB}) gives rise to the following evolution equations in the conditional means
$(\overline{U}_t,\overline{F}_t)$:
\begin{subequations} \label{eq:MFKB_mean}
\begin{align}
{\rm d}\overline{U}_t &= -\epsilon^{-1} \overline{U}_t{\rm d}t + \overline{F}_t{\rm d}t -
K_t^u (\overline{U}_t {\rm d}t - {\rm d}\mathcal{Y}_t), \\
{\rm d}\overline{F}_t &= -K_t^f (\overline{U}_t {\rm d}t - {\rm d}\mathcal{Y}_t).
\end{align}
\end{subequations}
The deviations $\Delta U_t:=U_t-\overline{U}_t$ and $\Delta F_t:=F_t - \overline{F}_t$ thus satisfy
\begin{subequations} \label{eq:MFKB_deviations}
\begin{align}
{\rm d}\Delta U_t &= - \epsilon^{-1}\Delta U_t{\rm d}t + \Delta F_t{\rm d}t + \gamma^{1/2}
{\rm d}W_t - \frac{1}{2} K_t^u \Delta U_t {\rm d}t,\\
{\rm d} \Delta F_t &= - \frac{1}{2} K_t^f \Delta U_t {\rm d}t.
\end{align}
\end{subequations}

We can further decompose (\ref{eq:MFKB_mean}) by introducing the \lq frequentist\rq\, expectation values 
$m_t^u = \mathbb{E}^\ast[\overline{U}_t]$ and  $m_t^f = \mathbb{E}^\ast[\overline{F}_t]$ 
with respect to the observation process (\ref{eq:obs_process_2}), 
and the deviations $\Delta \overline{U}_t := \overline{U}_t - m_t^u$ and 
$\Delta \overline{F}_t := \overline{F}_t - m_t^f$. Taking this expectation in (\ref{eq:model2}) yields
\begin{equation}
\mathbb{E}^\ast[{\rm d}\mathcal{Y}_t] = \mathbb{E}^\ast[\mathcal{U}_t] {\rm d}t.
\end{equation}
Introducing the shorthand $\mu_t = \mathbb{E}^\ast[\mathcal{U}_t]$, one is left with 
the ordinary differential equations
\begin{subequations} \label{eq:MFKB_mean_mean}
\begin{align} \label{eq:MFKB_mean_mean_a}
\frac{\rm d}{{\rm d}t} m_t^u &= -\epsilon^{-1} m_t^u + m_t^f - K_t^u (m_t^u - \mu_t),\\
\frac{\rm d}{{\rm d}t} m_t^f &= - K_t^f (m_t^u - \mu_t), \label{eq:MFKB_mean_mean_b}
\end{align}
\end{subequations}
for the mean values $(m_t^u,m_t^f)$.  Furthermore, defining $\Delta \mathcal{U}_t := \mathcal{U}_t - \mu_t$,
\begin{subequations} \label{eq:MFKB_mean_deviations_temp}
\begin{align} \label{eq:MFKB_mean_deviations_temp_a}
{\rm d}\Delta \overline{U}_t &= -\epsilon^{-1} \Delta \overline{U}_t{\rm d}t + \Delta \overline{F}_t{\rm d}t -
K_t^u \left((\Delta \overline{U}_t  - \Delta \mathcal{U}_t){\rm d}t - \rho^{1/2} {\rm d}\mathcal{V}_t \right), \\
{\rm d}\Delta \overline{F}_t &= -K_t^f \left((\Delta \overline{U}_t - \Delta \mathcal{U}_t){\rm d}t - \rho^{1/2} {\rm d}\mathcal{V}_t 
\right) 
\end{align}
\end{subequations}
for the deviations $(\Delta \overline{U}_t,\Delta \overline{F}_t)$. 
It follows from (\ref{eq:signal_process}) that the expectation value $\mu_t$ satisfies the evolution equation 
\begin{equation} \label{eq:signal_mean}
\frac{\rm d}{{\rm d}t} \mu_t = -\epsilon^{-1} \mu_t + \mathcal{F}^\ast,
\end{equation}
while the deviation $\Delta \mathcal{U}_t$ satisfies
\begin{equation} \label{eq:signal_deviations}
{\rm d}\Delta \mathcal{U}_t  = -\epsilon^{-1} \Delta \mathcal{U}_t{\rm d}t + \gamma^{1/2} {\rm d}\mathcal{W}_t .
\end{equation}
Both equations follow from (\ref{eq:signal_process}). We finally combine (\ref{eq:MFKB_mean_deviations_temp_a}) and
(\ref{eq:signal_deviations}) into a single equation for the new variable $\widehat{U}_t := \Delta \overline{U}_t - \Delta \mathcal{U}_t$, and replace (\ref{eq:MFKB_mean_deviations_temp}) by
\begin{subequations} \label{eq:MFKB_mean_deviations}
\begin{align}
{\rm d}\widehat{U}_t &= -\epsilon^{-1} \widehat{U}_t{\rm d}t + \Delta \overline{F}_t{\rm d}t - \gamma^{1/2}{\rm d}\mathcal{W}_t -
K_t^u \left(\widehat{U}_t  {\rm d}t - \rho^{1/2} {\rm d}\mathcal{V}_t \right), \\
{\rm d}\Delta \overline{F}_t &= -K_t^f \left(\widehat{U}_t{\rm d}t - \rho^{1/2} {\rm d}\mathcal{V}_t 
\right).
\end{align}
\end{subequations}

We have thus decomposed the mean-field Kalman--Bucy equations into the analysis of the three subsystems
(\ref{eq:MFKB_deviations}), (\ref{eq:MFKB_mean_mean}) and (\ref{eq:MFKB_mean_deviations}). Here, 
(\ref{eq:MFKB_mean_mean}) describes the systematic bias in the Bayesian mean estimator, while 
(\ref{eq:MFKB_mean_deviations}) and (\ref{eq:MFKB_deviations}) characterise the `frequentist'  and
`Bayesian' uncertainties, respectively.

We note that the equations in (\ref{eq:MFKB_deviations}) do not depend on the observation process $\mathcal{Y}_{(0,t]}$ 
that they are decoupled from (\ref{eq:MFKB_mean_mean}) and 
(\ref{eq:MFKB_mean_deviations}). In fact, since $(\Delta U_t,\Delta F_t)$ are centred Gaussian random variables, 
it is sufficient to look at the (deterministic) time evolution equations for the posterior variances
\begin{equation}
\sigma_t^{u} := \mathbb{E}[\Delta U_t^2],  \quad
\sigma_t^{f} := \mathbb{E}[\Delta F_t^2] ,
\end{equation}
and the posterior covariance
\begin{equation}
\sigma_t^{uf} := \mathbb{E}[\Delta U_t\Delta F_t] ,
\end{equation}
which are given by
\begin{subequations} \label{eq:MFKB_deviations_2}
\begin{align} \label{eq:MFKB_derivations_21}
\frac{\rm d}{{\rm d}t} \sigma_t^{u} &= -2\epsilon^{-1} \sigma_t^{u} + 2\sigma_t^{uf} +\gamma
- \rho^{-1} \left(\sigma_t^{u}\right)^2,\\ \label{eq:MFKB_derivations_22}
\frac{\rm d}{{\rm d}t} \sigma_t^{uf} &= -\epsilon^{-1} \sigma_t^{uf} + \sigma_t^{f} 
- \rho^{-1} \sigma_t^{u} \sigma_t^{uf},\\
\frac{\rm d}{{\rm d}t} \sigma_t^{f} &= - \rho^{-1} \left(\sigma_t^{uf}\right)^2.
\end{align}
\end{subequations}
Since $(\widehat{U}_t,\Delta \overline{F}_t)$ are also centred Gaussian random variables, 
the time-dependent linear SDEs
(\ref{eq:MFKB_mean_deviations}) can again be analyzed in terms of the variances
\begin{equation}
p_t^{u} := \mathbb{E}^\ast [\widehat{U}_t^2], \quad
p_t^{f} := \mathbb{E}^\ast [\Delta \overline{F}_t^2 ],
\end{equation}
and the covariance
\begin{equation}
p_t^{uf} := \mathbb{E}^\ast [\widehat{U}_t\Delta \overline{F}_t ].
\end{equation}
These quantities satisfy the linear time-dependent ordinary differential equations
\begin{subequations} \label{eq:MFKB_mean_deviations_2}
\begin{align}
\frac{\rm d}{{\rm d}t} p_t^u &= -2\epsilon^{-1} p_t^u + 2 p_t^{uf} + \gamma - 2K_t^u p_t^u + \rho (K_t^u)^2,\\
\frac{\rm d}{{\rm d}t} p_t^{uf} &= -\epsilon^{-1} p_t^{uf} + p_t^f - K_t^u p_t^{uf} - K_t^f p_t^u + \rho K_t^u K_t^f,\\
\frac{\rm d}{{\rm d}t} p_t^f &= -2K_t^f p_t^{uf} + \rho (K_t^f)^2.
\end{align}
\end{subequations}
Equations \eqref{eq:MFKB_deviations_2} and \eqref{eq:MFKB_mean_deviations_2} are combined to 
to yield
\begin{subequations} \label{eq:MFKB_mix}
\begin{align} 
\frac{\rm d}{{\rm d}t} \Delta p_t^u &= -2(\epsilon^{-1} + K_t^u)\Delta p_t^u + 2\Delta p_t^{uf} ,\\
\frac{\rm d}{{\rm d}t} \Delta p_t^{uf} &= -(\epsilon^{-1}+K_t^u)\Delta p_t^{uf}+ \Delta p_t^f
 - K_t^f \Delta p_t^u ,\\ \
\frac{\rm d}{{\rm d}t} \Delta p_t^f &= -2K_t^f \Delta p_t^{uf}
\end{align}
\end{subequations}
in the variables
\begin{equation}
\Delta p_t^{uf} := p_t^{uf} - \sigma_t^{uf},\quad
\Delta p_t^{u} := p_t^{u} - \sigma_t^{u},\quad \Delta p_t^{f} := p_t^{f} - \sigma_t^{f}.
\end{equation}
The initial conditions are $p_0^u = p_0^f=p_0^{uf} = \sigma_0^u = \sigma_0^{uf}= 0$, and $\sigma^f_0 = \epsilon^{\alpha + 1/2}$.

\begin{lemma} \label{lemma:asymptotic}
It holds that
\begin{equation}
p_t^u \to \sigma_t^u, \quad p_t^f \to \sigma_t^f, \quad p_t^{uf} \to \sigma_t^{uf}
\end{equation}
as $t\to \infty$, that is, $\Delta p_t^u = \Delta p_t^{uf}=\Delta p_t^f =0$ is a stable equilibrium point
of (\ref{eq:MFKB_mix}). 
\end{lemma}

\begin{proof}
The lemma follows from the decay property of
\begin{equation}
V_t(\Delta p_t^u,\Delta p_t^{uf},\Delta p_t^f) 
:= \frac{1}{2} (\Delta p_t^f)^2 + K_t^f (\Delta p_t^{uf})^2 + \frac{(K_t^f)^2}{2} (\Delta p_t^u)^2
\end{equation}
and the fact that $K_t^f \sim c/t$ for $t$ sufficiently large with an appropriate constant $c>0$.
In particular,
\begin{equation}
\partial_t V_t = \left\{ (\Delta p_t^{uf})^2 + K_t^f (\Delta p_t^u)^2\right\} \frac{\rm d}{{\rm d}t} K_t^f  \le 0,
\end{equation}
and for the gradient $\nabla V_t(\Delta p_t^u,\Delta p_t^{uf},\Delta p_t^f) \in \mathbb{R}^3$:
\begin{equation} \label{eq:lemma_proof1}
\nabla V_t \cdot \frac{\rm d}{{\rm d}t} \left(\begin{array}{c}
\Delta p_t^u \\ \Delta p_t^{uf} \\ \Delta p_t^f \end{array}\right) =
-2(\epsilon^{-1}+K_t^u) \left\{ K_t^f  (\Delta p_t^{uf})^2 + (\Delta p_t^u)^2\right\} K_t^f \le 0.
\end{equation}
Therefore, the total time derivative satisfies
\begin{equation}
\frac{\rm D}{{\rm D}t} V_t( \Delta p_t^u,\Delta p_t^{uf},\Delta p_t^f) \le 0.
\end{equation}
Furthermore, for large enough $t$, 
\begin{equation}
p_t^{uf} \approx \left( \epsilon^{-1} + K_t^u \right)^{-1} p_t^f, \qquad p_t^u \approx 
 \left( \epsilon^{-1} + K_t^u \right)^{-1} p_t^{uf} \approx  \left( \epsilon^{-1} + K_t^u \right)^{-2} p_t^f, 
\end{equation}
and $\partial_t V_t$ becomes small relative to (\ref{eq:lemma_proof1}) since $\frac{\rm d}{{\rm d}t} K_t^f \sim -c/t^2$. 
Hence, for $t$ again sufficiently large and with an appropriate constant $C>0$,
\begin{equation}
\frac{\rm d}{{\rm d}t}  (\Delta p_t^f)^2 \approx -\frac{C}{t}  (\Delta p_t^f)^2,
\end{equation}
and the desired convergence result follows.
\end{proof}

\noindent
Lemma \ref{lemma:asymptotic} implies that, as for the pure drift estimation problem, 
the Bayesian (filtering) variance asymptotically covers the \lq frequentist\rq{} (data) variance.

\subsubsection{The large-$k$ Fourier mode case} \label{sec:single_mode_large}

We now investigate the case $\epsilon \ll 1$ in more detail.
The equations (\ref{eq:MFKB_deviations_2}) 
possess a slow manifold for $\epsilon>0$ sufficiently small \citep{verhulst07,Shchepakina2014SingularPI}. To leading order,  solutions on the slow manifold satisfy
\begin{equation} \label{eq:slow_manifold_approx}
\sigma_t^{u} = \epsilon \frac{\gamma}{2} , \qquad \sigma_t^{uf} = \epsilon \sigma_t^{f}.
\end{equation}
In other words, the long-time variance of $U_t$ is governed by $\sigma_t^{u} = \epsilon \gamma /2$, 
and of $F_t$ by
\begin{equation} \label{eq:LT_deviations}
\frac{\rm d}{{\rm d}t} \sigma_t^{f} = - \mathcal{K}_t \sigma_t^{f}
\end{equation}
up to higher-order terms in $\epsilon$. Here we have introduced the new Kalman gain factor
\begin{equation}
\mathcal{K}_t := \frac{\epsilon^2}{\rho} \sigma_t^{f}.
\end{equation}
We also notice that the Kalman gain factor $K_t^u$ satisfies $K_t^u = \mathcal{O}(\epsilon)$. 
See the Appendix for more details on the derivation of the reduced system (\ref{eq:LT_deviations}).

The combined time-dependent linear equations (\ref{eq:signal_mean}) and (\ref{eq:MFKB_mean_mean}) 
also give rise to a slow manifold, with the slow dynamics governed by $\mu_t = \epsilon \mathcal{F}^\ast$, $m^u_t = \epsilon m^f_t$, and
\begin{equation} \label{eq:LT_mean_mean}
\frac{\rm d}{{\rm d}t} m^f_t = - \mathcal{K}_t (m^u_t - \mathcal{F}^\ast).
\end{equation}
The initial condition is given by $m^f_0 = 0$.

The solution of (\ref{eq:LT_deviations}) is explicitly given by
\begin{equation} \label{eq:estimate_variance}
\sigma_t^{f} = \frac{\rho \sigma_0^{f}}{\epsilon^2 \sigma_0^{f} t + \rho}.
\end{equation}
We also find that the solution of (\ref{eq:LT_mean_mean}) satisfies
\begin{equation} \label{eq:estimate_bias}
m^f_t = \left( 1 - \frac{\sigma_t^{f}}{\sigma_0^{f}}\right) \mathcal{F}^\ast.
\end{equation}
The large-time limit yields that $m^f_t - \mathcal{F}^\ast$ decays with rate $t^{-1}$.
We have already shown that $p_t^{f} \to \sigma_t^{f}$ as $t\to \infty$.

\begin{remark} We note that the combined state and drift estimation problem under the slow manifold approximation behaves exactly like the stationary drift estimation problem from Section \ref{sec:drift}. In particular, compare equations (\ref{eq:sigma_f_drift}) and (\ref{eq:mean_f_drift}). The numerical experiments in Section \ref{sec:numerics} reveal that
the slow manifold approximation already holds for rather small-$k$ Fourier modes. 
\end{remark}

\noindent
In summary, we have shown that
\begin{equation} \label{eq:drift_probability}
\overline{F}_t - \mathcal{F}^\ast = \mathcal{O}_{\rm P}(\Delta_t)
\end{equation}
with 
\begin{equation}
\Delta_t \sim{\rm N}(m^f_t-\mathcal{F}^\ast,p_t^f).
\end{equation}
Hence, according to Lemma 8.2.~from \cite{Ghosal2017book}, the posterior contraction rate 
$\phi_t$ in
\begin{equation} \label{eq:contraction_rate}
\mathbb{E}^\ast \left[\mathbb{P} \left\{F_t: |F_t-\mathcal{F}^\ast| \ge M_t \phi_t \,| \,\mathcal{Y}_{(0,t]} \right\} \right] \xrightarrow{t \to \infty} 0
\end{equation}
is provided by
\begin{equation}
\phi_t^2   =  \sigma_t^f + p_t^f + (m^f_t-\mathcal{F}^\ast)^2  \asymp t^{-1}
\end{equation}
since $p_t^f < \sigma_t^f$. 

\begin{remark} While the posterior contraction rate is $\phi_t \asymp t^{-1/2}$ for any fixed $k \gg 1$, the analysis of the infinite-dimensional state and parameter estimation problem is more complicated since one takes the limit $k\to \infty$ first, and then the limit $t \to \infty$. 
\end{remark}

\subsubsection{The small-$k$ Fourier mode case} \label{sec:single_mode_small}

The above analysis holds for $\epsilon$ sufficiently small, that is, for sufficiently large Fourier modes. 
We now investigate the behaviour of the filtering 
and estimation problem for $\epsilon = \mathcal{O}(1)$. We set $\epsilon = 1$, that is, $k=1$ for
simplicity. We start again with (\ref{eq:MFKB_deviations_2}) and note that there is a finite time $t_\ast$ such that 
$\sigma_t^{uf}$ has become sufficiently small relative to $\gamma$, and we conclude from 
(\ref{eq:MFKB_derivations_21}) that
\begin{equation} \label{eq:approx_sigma_u}
\sigma_t^u \approx C_0 \qquad \mbox{for all} \qquad t \ge t^\ast ,
\end{equation}
where $C_0>0$ satisfies the quadratic equation
\begin{equation}
2C_0 - \gamma + \rho^{-1} C_0^2=0.
\end{equation} 

\begin{remark} We note that $(\sigma_t^{uf})^2 \le \sigma_t^u \sigma_t^f$. We also
know that $\sigma_t^u$ remains bounded since the system is fully observed, and that $\sigma_t^f$ decays. Therefore,
the assumption that $\sigma_t^{uf}$ becomes small relative to $\gamma$ is justified for $t\ge t_\ast$.
\end{remark} 

Upon substituting (\ref{eq:approx_sigma_u}) into (\ref{eq:MFKB_derivations_22}),  we can conclude an exponential decay of $\sigma_t^{uf}$ towards $\sigma_t^f$ and, hence, we can choose $t_\ast$ large enough
such that 
\begin{equation}
\sigma_t^{uf} \approx \left(1 + \frac{C_0}{\rho}\right)^{-1} \sigma_t^f \qquad \mbox{for all}\qquad t\ge t_\ast .
\end{equation}
In other words, the asymptotic dynamics in $\sigma_t^f$ is again governed by (\ref{eq:LT_deviations}) with 
Kalman gain factor 
\begin{equation} \label{eq:Kf_approx}
\mathcal{K}_t := C_1^2 \frac{\sigma_t^f}{\rho} 
\end{equation}
and constant 
\begin{equation}
C_1 := \left(1 + \frac{C_0}{\rho} \right)^{-1}.
\end{equation}
Therefore, $\sigma_t^f$ behaves asymptotically like $t^{-1}$.

We next analyse the long-time dynamics of (\ref{eq:MFKB_mean_mean}). We use $K_t^u \approx C_0/\rho$
and find that (\ref{eq:MFKB_mean_mean_a}) gives rise to the quasi-equilibrium 
\begin{equation}
0 \approx -m^u_t + m^f_t - \frac{C_0}{\rho}(m^u_t - \mathcal{F}^\ast) =
-C_1^{-1} (m^u_t-\mathcal{F}^\ast) + m^f_t - \mathcal{F}^\ast
\end{equation}
for $t\ge t_\ast$ and $t_\ast$ sufficiently large. Upon substituting this relation into (\ref{eq:MFKB_mean_mean_b}) and using 
\begin{equation} \label{eq:Kf_approx_2}
K_t^f \approx C_1 \frac{\sigma_t^f}{\rho} = C_1^{-1}\mathcal{K}_t,
\end{equation}
(\ref{eq:Kf_approx}) as well as $\mu_t \approx \mathcal{F}^\ast$, we arrive at (\ref{eq:LT_mean_mean}) with
gain factor $\mathcal{K}_t$ given by (\ref{eq:Kf_approx}).
Repeating the analysis from Section \ref{sec:single_mode_large}, 
one finds again that $\overline{F}_t-\mathcal{F}^\ast$ satisfies (\ref{eq:drift_probability}) with $|m_t^f-\mathcal{F}^\ast|
\asymp t^{-1}$ and $p_t^u < \sigma_t^u$.

Summarising our findings, we conclude that the contraction rate $\phi_t$ in (\ref{eq:contraction_rate}) is of order $t^{-1/2}$ for $t$ sufficiently large.

%
\subsection{Asymptotic rates for the stochastic heat equation}
%

We recall our Assumption \ref{assumption1} on the asymptotic behaviour of the true drift function $\mathcal{F}^\ast(k)$ and on the prior random variables $F_0(k)$, $k\ge 1$. These assumptions together with the results from Sections \ref{sec:single_mode_large} and \ref{sec:single_mode_small}, and in particular (\ref{eq:estimate_variance}) and (\ref{eq:estimate_bias}) with $\epsilon = k^{-2}$, imply that
\begin{equation} \label{eq:theorem2_1}
\sigma_t^f(k) \sim \frac{\rho k^{-(2\alpha+1)}}{k^{-(2\alpha+5)} t + \rho}
\end{equation}
and
\begin{equation} \label{eq:theorem2_2}
|m^f_t(k) -\mathcal{F}^\ast(k)| \sim \frac{\rho k^{-(\beta+1/2)}}{k^{-(2\alpha+5)} t + \rho},
\end{equation}
for $k\ge k_\ast$, and
\begin{equation}
\sigma_t^f(k) \asymp t^{-1}, \qquad 
m^f_t(k) -\mathcal{F}^\ast(k) \asymp t^{-1}
\end{equation}
for all $k\le k_\ast$, and $t\ge t_\ast$ sufficiently large. Here $k_\ast$ and $t_\ast$ are chosen sufficiently large such that the analysis of Sections \ref{sec:single_mode_large} and \ref{sec:single_mode_small}, respectively, applies. 

\begin{theorem} \label{theorem2}
Under Assumption \ref{assumption1} and the regularity of the true drift function $\mathcal{F}^\ast(k)$ and the prior $F_0(k)$, $k\ge 1$, the asymptotic contraction rate (\ref{eq:l2_posterior_contraction_rate}) in the $l_2$-norm becomes
\begin{equation}
  \phi(t)^2 = \sum_{k\ge 1} \left( (m^f_t(k)-\mathcal{F}^\ast(k))^2 + p_t^f(k) + \sigma_t^f(k)\right) \asymp 
  t^{-2\,\mathrm{min}\{\alpha,\beta\}/(2\alpha + 5)} .
\end{equation}
\end{theorem}

\begin{proof} The theorem follows from (\ref{eq:theorem2_1}) and (\ref{eq:theorem2_2}) together with $p_t^f(k) <
\sigma_t^f(k)$ and standard asymptotic estimates for infinite sums. See, for example, again Lemma K.7 from
\cite{Ghosal2017book}.
\end{proof}

We note that the rate in Theorem \ref{theorem2} is the same as in (\ref{eq:contraction_rate_pure_drift}) for the stationary drift estimation problem. In other words, the need for estimating the states as well as the drift does not deteriorate the asymptotic contraction rates. In fact, the data does not affect the posterior uncertainty in the states as $k\to \infty$, which is entirely determined by the equilibrium distribution of the associated Ornstein--Uhlenbeck process. We verify this behaviour through a simple numerical experiment in the following section.

%
\section{Numerical exploration} \label{sec:numerics}
%

\begin{figure}
\begin{center}
\includegraphics[width=0.45\textwidth]{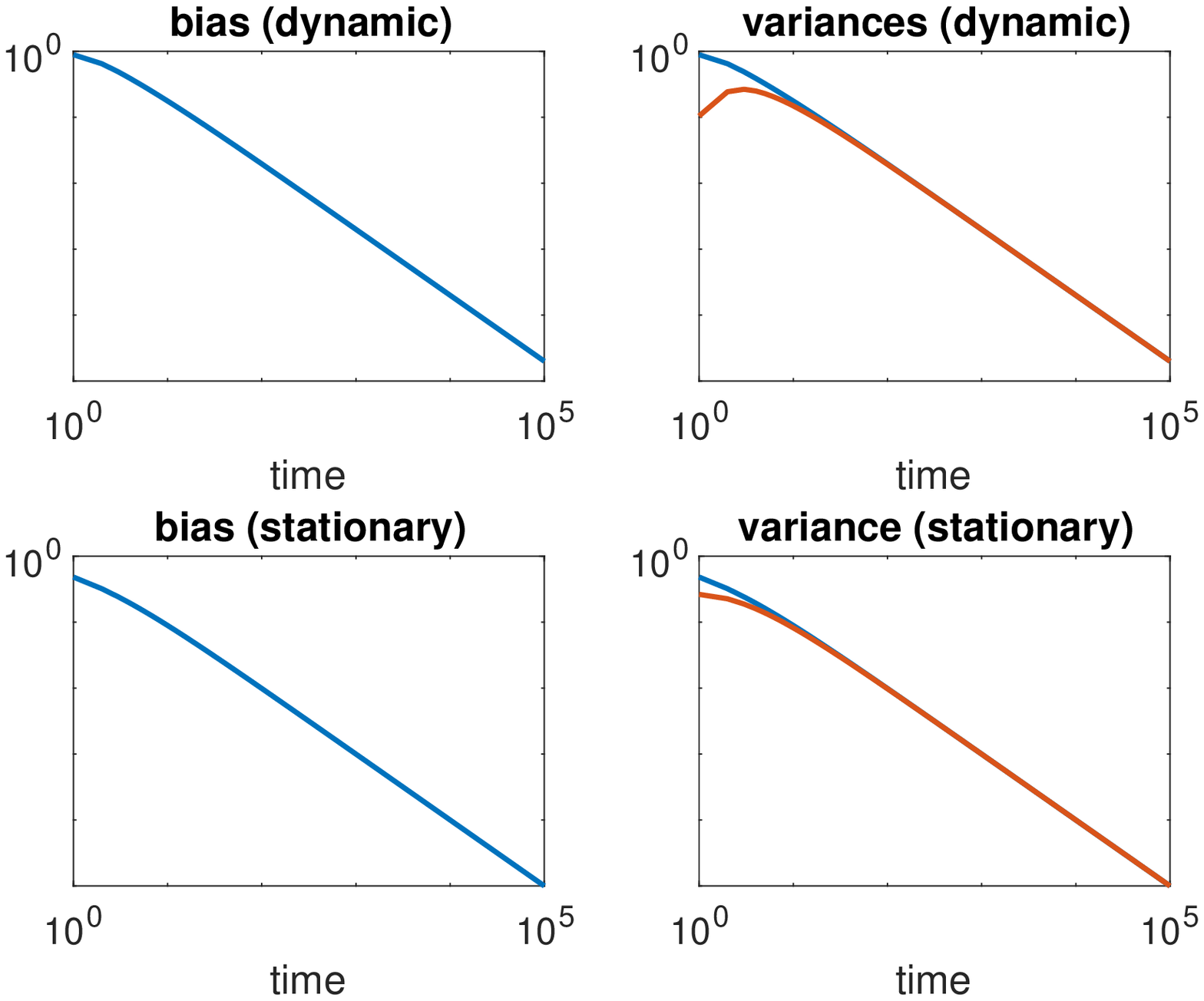} $\quad$
\includegraphics[width=0.45\textwidth]{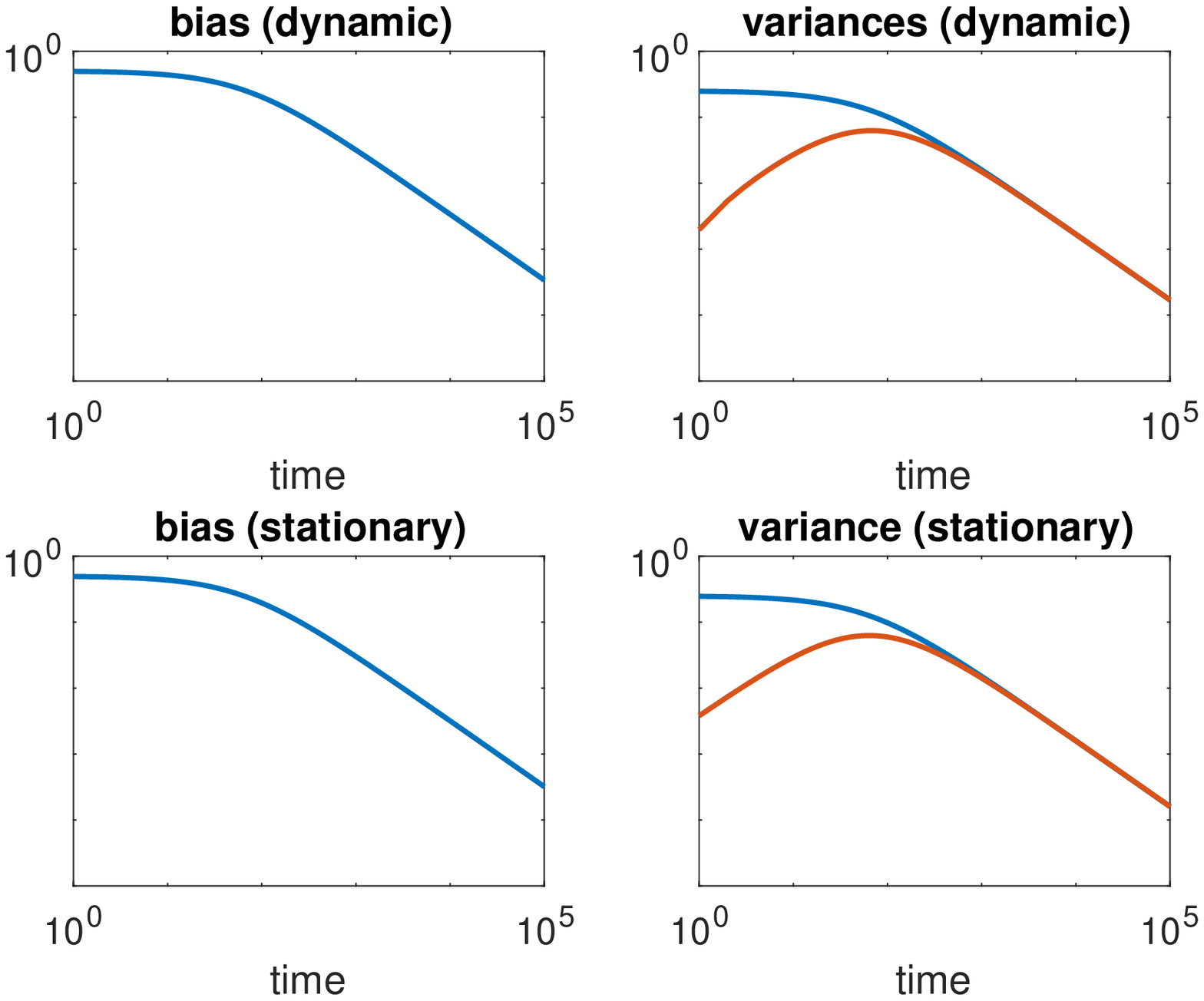}\\ \smallskip 
(a) $k=1$ $\qquad\qquad\qquad\qquad\qquad \qquad \qquad $ (b) $k=2$\\ \medskip
\includegraphics[width=0.45\textwidth]{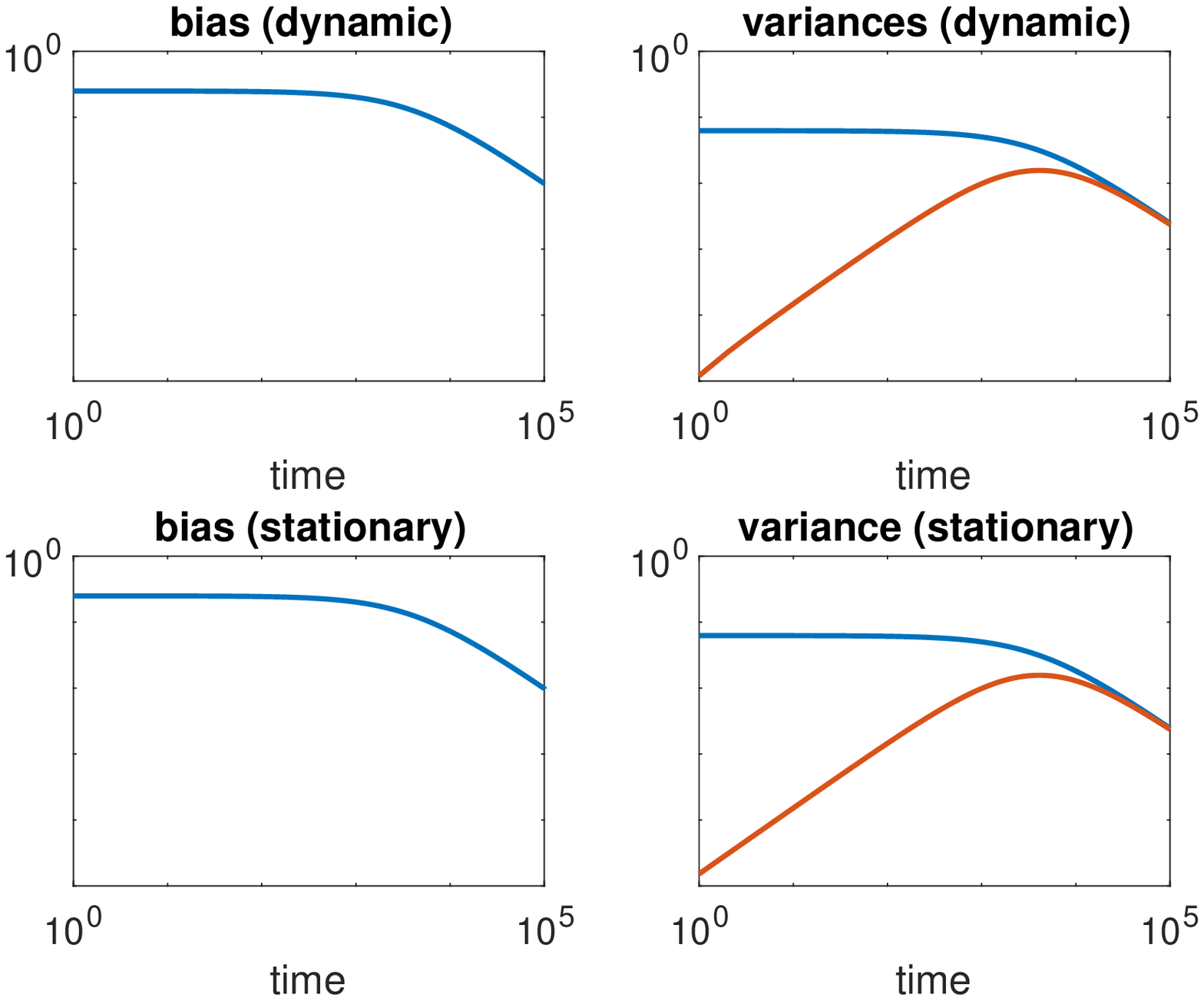} $\quad$
\includegraphics[width=0.45\textwidth]{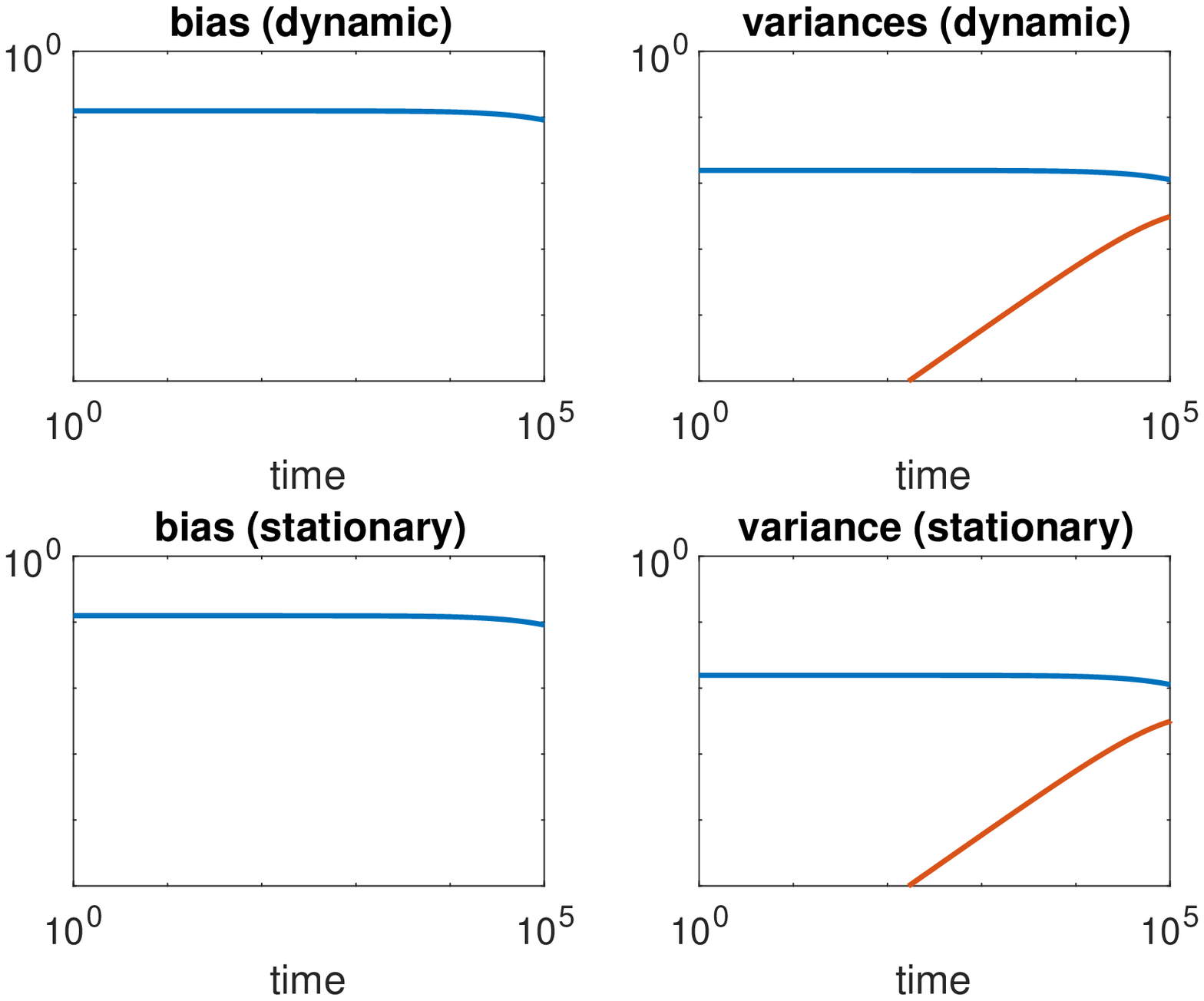}\\ \smallskip 
(c) $k=4$ $\qquad\qquad\qquad\qquad\qquad \qquad \qquad $ (d) $k=8$
\end{center}
\caption{We display the time evolution of the bias $|m_t^f(k)-\mathcal{F}^\ast(k)|$ and the variances
$\sigma_t^f(k)$ and $p_t^f(k)$ for increasing values of $k$. We compare the results from the combined
state and parameter estimation problem (labelled \lq dynamic\rq ) with those from the associated direct inference problem
in the parameter alone (labelled \lq stationary\rq ). The delay in the onset of the asymptotic $t^{-1}$ regime as $k$ 
increases can be clearly seen as well as that $p_t^f(k) < \sigma_t^f(k)$ for all $t>0$. Furthermore, as theoretically predicted, the dynamic and stationary problem formulations behave almost identically in terms of their drift estimates.}
\label{figure1}
\end{figure}

We numerically implemented the evolution equations (\ref{eq:MFKB_mean_mean}), (\ref{eq:MFKB_deviations_2}), and (\ref{eq:MFKB_mean_deviations_2}) for the Bayesian means, the 
Bayesian variances, and the frequentist variances, respectively, for different Fourier modes $k$. We also implemented the corresponding equations
from Section \ref{sec:drift} for the pure drift estimation problem. Simulations were run with $\rho = \gamma = 1$.
The true reference value was set to $\mathcal{F}^\ast (k) = k^{-\beta - 1/2}$ with $\beta = 1/2$, and we used $\alpha = 1/2$ for
the prior variance (\ref{eq:prior_f}). 

The results can be found in Figure \ref{figure1}. They reveal a very similar behaviour for the combined state
and drift estimation and the pure drift estimation problems. The delay in the onset of the asymptotic $t^{-1}$
regime as $k$ increases can also be seen. Overall, this simple numerical experiment confirms our theoretical
investigations with regard to the dynamical behaviour in each Fourier mode from Section \ref{sec:state_drift}.

%
%
\section{Conclusions} \label{sec:conclusions}

We have provided an analysis of the infinite dimensional Kalman--Bucy filter mean-field equations (\ref{eq:MFKB}) for a combined state and drift estimation problem defined in spectral space by (\ref{eq:model1}) and (\ref{eq:model2}). The derived asymptotic posterior contraction rates in the unknown drift function $f^\ast$ from Theorem \ref{theorem2} agree with those derived for an associated stationary problem formulation in Theorem \ref{theorem1}. These theoretical findings imply that the required additional estimation of the states does not lead to a deterioration of the contraction rates, which has been confirmed numerically in Section \ref{sec:numerics}. An extension to nonlinear and non-Gaussian estimation problems for SPDEs and different types of observation processes can be envisioned through nonlinear extensions of the Kalman--Bucy mean-field equations. See, for example, \cite{nrr2019}. Furthermore, our results can be combined with those from  \cite{GNSUl2017} to study the coverage probabilities of Bayesian credible sets in a non-asymptotic regime. One could also discuss adaptive choices for the prior $F_0(k)$. See, for example, \cite{knapik2016}.

\section*{Acknowledgements} This research has been partially funded by Deutsche Forschungsgemeinschaft (DFG) - SFB1294/1 - 318763901.

\section*{Appendix: Slow manifold approximation} \label{appendix}

Given a system of differential equations of the form
\begin{subequations} \label{eq:slow-fast}
\begin{align}
\frac{\rm d}{{\rm d}t} x_t &= - \epsilon^{-1} A x_t + f(x_t,y_t;\epsilon),\\
\frac{\rm d}{{\rm d}t} y_t &= g(x_t,y_t;\epsilon) \label{eq:slow-fast_b},
\end{align}
\end{subequations}
with $\epsilon>0$ sufficiently small and $A$ a symmetric positive definite matrix, there exists a smooth manifold 
$\mathcal{M}_\epsilon$ which is exponentially attractive and invariant under the dynamics of (\ref{eq:slow-fast}). 
See, for example, \cite{verhulst07,Shchepakina2014SingularPI} for details. 

Approximations of the slow manifold can be found by utilising the principle of bounded derivatives, that is, 
\begin{equation} \label{eq:bounded_derivative}
\frac{{\rm d}^n}{{\rm d}t^n} x_t = \mathcal{O}(\epsilon^0)
\end{equation}
for $n\ge 1$. More specifically, (\ref{eq:bounded_derivative}) with $n=1$ implies that
\begin{equation}
x_t^{(0)} = \mathcal{O}(\epsilon),
\end{equation}
while the same equation with $n=2$ leads to
\begin{equation}
\frac{{\rm d}}{{\rm d}t} x_t^{(1)} = \mathcal{O}(\epsilon)
\end{equation}
and, hence,
\begin{equation} \label{eq:slow_manifold}
x_t^{(1)} = \epsilon A^{-1} f(0,y_t;0) + \mathcal{O}(\epsilon^2) .
\end{equation}
Here the superscript indicates the order of the approximation.
Substituting the leading order term of (\ref{eq:slow_manifold}) into (\ref{eq:slow-fast_b}) provides a reduced dynamics 
in the $y$-variable alone. 
The next order correction is obtained from
\begin{equation}
\frac{\rm d^2}{{\rm d}t^2} x_t^{(2)} = -\epsilon^{-1} A \frac{\rm d}{{\rm d}t} x_t^{(2)} + 
D_1 f(x_t^{(2)},y_t;\epsilon)\frac{\rm d}{{\rm d}t} x_t^{(2)} + D_2 f(x_t^{(2)},y_t;\epsilon)\frac{\rm d}{{\rm d}t}y_t =
\mathcal{O}(\epsilon)
\end{equation}
and yields
\begin{equation}
x_t^{(2)} = \epsilon A^{-1} f(x_t^{(1)},y_t;\epsilon) + \epsilon^2 A^{-2} D_2 f(0,y_t;0) 
\frac{\rm d}{{\rm d}t} y_t + \mathcal{O}(\epsilon^3).
\end{equation}
Higher-order approximations to $\mathcal{M}_\epsilon$ can be obtained by exploiting (\ref{eq:bounded_derivative}) 
for $n> 3$. 

Note that (\ref{eq:MFKB_deviations_2}), for example, fits into the framework (\ref{eq:slow-fast}) 
with $x = (\sigma^u,\sigma^{uf})^{\rm T}$, $y = \sigma^f$, and 
\begin{equation}
A = \left( \begin{array}{cc} 2 &0\\0 & 1 \end{array} \right).
\end{equation}
Hence (\ref{eq:slow_manifold}) results in (\ref{eq:slow_manifold_approx}) and the evolution equation 
(\ref{eq:LT_deviations}) in the slow variable $y = \sigma^f$. The induced approximation error remains of order $\epsilon$ 
over time intervals of order $\epsilon^{-2}$.

What about $t\gg \epsilon^{-2}$? In this large time limit, the solution (\ref{eq:estimate_variance}) 
of (\ref{eq:LT_deviations}) satisfies
\begin{equation}
\frac{\rm d}{{\rm d}t} \sigma_t^f = \mathcal{O}(\epsilon^{-2} t^{-2})
\end{equation}
implying 
\begin{equation}
\frac{\rm d}{{\rm d}t} \sigma_t^{uf} = \mathcal{O}(\epsilon^{-1} t^{-2})
\end{equation}
along solutions on the slow manifold. In other words, it follows from (\ref{eq:MFKB_derivations_22}) that 
\begin{equation}
(1+\epsilon \rho^{-1} \sigma_t^u)\sigma_t^{uf} + \epsilon \sigma_t^f = \mathcal{O}(t^{-2})
\end{equation}
for $t \gg \epsilon^{-2}$. The resulting correction terms to the first-order balance relation 
$\sigma_t^{uf} = \epsilon \sigma_t^f$ do not alter the qualitative long-time solution behaviour of
the reduced slow equations (\ref{eq:LT_deviations}) as verified by the numerical results from Section 
\ref{sec:numerics}, where a $t^{-1}$ long-time decay rate in $\sigma_t^f$ has been observed.

%
%
\bibliographystyle{agsm}
\bibliography{bib_paper}
%
%


\end{document}